\documentclass[10pt,oneside,a4paper]{amsart}
\usepackage{latexsym}
\usepackage{amsfonts,amsmath,amssymb,indentfirst}
\usepackage[english]{babel}
\usepackage[all]{xy}
\usepackage[colorlinks=true,pagebackref,hyperindex]{hyperref}

\newcommand{\bdism}{\begin{displaymath}}
\newcommand{\edism}{\end{displaymath}}
\newcommand{\cc}{\mathbb{C}}
\newcommand{\rr}{\mathbb{R}}
\newcommand{\qq}{\mathbb{Q}}

\newcommand{\pp}{\mathbb{P}}
\newcommand{\oo}{\mathcal{O}}

\newcommand{\T}{\mathcal{T}}
\newcommand{\D}{\mathcal{D}}

\DeclareMathOperator{\vol}{vol}

\DeclareMathOperator{\Eff}{\overline{Eff}}
\DeclareMathOperator{\Bag}{Big}
\DeclareMathOperator{\Nef}{Nef}

\newtheorem{theorem}{Theorem}[section]

\newtheorem{corollary}[theorem]{Corollary}
\newtheorem{lemma}[theorem]{Lemma}
\newtheorem{remark}[theorem]{Remark}
\newtheorem{definition}[theorem]{Definition}

\address{Department of Mathematics, Columbia University, New York NY 10027, USA} 
\email{dicerbo@math.columbia.edu}

\author{\scshape Gabriele Di Cerbo}
\title{\bf On Fujita's spectrum conjecture}
\begin{document}
\pagestyle{headings}
\begin{abstract}
We prove Fujita's spectrum conjecture on the discreteness of pseudo-effective thresholds for polarized varieties.
\end{abstract}
\maketitle

\tableofcontents

\section{Introduction}
\pagenumbering{arabic}

Let $X$ be a smooth projective variety over $\cc$. The minimal model program has set a now standard way to study the geometry of $X$ through the use of several cones of divisors associated to the variety. For example, if we want to understand varieties with non-nef canonical class, we consider the cone of nef divisors $\Nef(X)$ and we try to estimate how far is the canonical class of $X$ from the cone. A way of measuring that is given by the nef threshold of an ample divisor $H$ with respect to $X$. Roughly speaking, we fix a rational point in the interior of $\Nef(X)$ and we study when the line connecting the point and the class of $K_X$ meets the boundary of $\Nef(X)$. Letting the ample divisor $H$ vary, we can deduce important information on some parts of $\Nef(X)$. This is the idea behind the Cone Theorem, which is basically equivalent to the statement that the nef threshold is always a rational number with bounded numerator.

\begin{theorem}\label{cone}\cite[Theorem 3.5]{KM}
Let $X$ be a smooth projective variety and let $H$ be an ample divisor. Let $n(X,H)=\inf\{t\geq 0 \: | \: K_X +t H \in \Nef(X) \}$. Then $n(X,H)$ is a rational number with numerator bounded by $\dim(X)+1$.
\end{theorem} 

In particular, the set of nef thresholds is finite away from $0$. Theorem \ref{cone} is a fundamental result in birational geometry and it is the start of the minimal model program.  We refer to \cite{KM} for a proof of the theorem and much more on the subject.

For a different set of problems, it is more useful to understand the geometry of the cone of pseudo-effective divisors $\Eff(X)$. It is defined as the closure of the convex cone spanned by the classes of all effective $\rr$-divisors in $N^{1}(X)_{\rr}$. Moreover, the interior of $\Eff(X)$ is the cone of big divisors $\Bag(X)$, i.e. divisors with maximal Kodaira dimension. We can similarly define the pseudo-effective threshold to be $\tau (X,H)=\inf\{t\geq 0 \: | \: K_X +t H \in \Eff(X) \}$. It is again a way of measuring how far is the canonical class from the cone of pseudo-effective divisors and in particular, it is of interest only for varieties with not pseudo-effective canonical class, or equivalently for uniruled varieties.

Fujita realized that in order to study the adjunction theory for polarized varieties we need an analogue of Theorem \ref{cone} for the cone of pseudo-effective divisors. In his famous papers \cite{Fujita1} and \cite{Fujita2}, he conjectured that the set of pseudo-effective thresholds should behave in a similar way to the set of nef thresholds and, more precisely, it should be finite away from $0$ as well. This is known as Fujita's spectrum conjecture and its proof is the main result of this paper. 

\begin{theorem}\label{fujita}
Fix $n\geq 1$. Let $S_{n}$ be the set of pseudo-effective thresholds $\tau(X,H)$ of an ample divisor $H$ with respect to a smooth variety $X$. Then $S_{n}\cap [\epsilon, \infty)$ is a finite set for any $\epsilon >0$.  
\end{theorem}

The proof of the above theorem relies on the new results obtained by Birkar in \cite{Bir16} on boundedness of certain Fano varieties and an appropriate modification of the techniques in \cite{DiC12}. 

It seems an extremely hard problem to effectively bound the numerator of the pseudo-effective thresholds as in Theorem \ref{cone}. Moreover, since we are dealing with $\Eff(X)$, it seems natural to relax the positivity assumption and work with just a big divisor $H$. We show that in this case the discreteness of the set fails. 

\begin{theorem}\label{dense}
Fix $n\geq 3$. Then the set of pseudo-effective thresholds $\tau(X,H)$ of a big divisor $H$ with respect to a smooth variety $X$ is dense in $\rr_{\geq 0}$.
\end{theorem} 

It would be interesting to understand what kind of real numbers appear as pseudo-effective thresholds of big divisors. This should be connected to the results and conjectures in \cite{Kuronya}. Moreover, it is not clear if Theorem \ref{dense} holds for surfaces. The existence of the Zariski decomposition and a very general minimal model program for surfaces, established by Fujino in \cite{Fujino}, imply strong restrictions on the set of pseudo-effective thresholds. See Corollary \ref{surf} for more details. 

The set $S_n$ appearing in Theorem \ref{fujita} has been extensively studied. In \cite{BCHM}, it is shown that $S_n$ is contained in the rational numbers. In \cite{DiC12}, the author proved that there are no increasing sequences in $S_n$. Already those theorems have many applications and they provide important boundedness and effective birationality results, see \cite{Gabriele} and \cite{HMX} for more details. Finally, Theorem \ref{fujita} was already known for toric varieties \cite{Paf}.

The dual statement of Theorem \ref{fujita} should allow us to prove a cone theorem for movable curves as the one in \cite{Araujo} and \cite{Lehman}. It would be interesting to understand if Theorem \ref{fujita} implies a discreteness statement for extremal rays of the cone of movable curves. Moreover, Theorem \ref{fujita} has applications to number theory as well. For example, in \cite{Batyrev} and \cite{Yuri}, Fujita's spectrum conjecture is used to compute the asymptotic of the number of rational points with bounded height on projective varieties. 

\vspace{0.5cm}

\noindent\textbf{Acknowledgements}. I would like to thank Chen Jiang, Chenyang Xu and the referee for many constructive comments.
\section{Preliminaries}

Even if the statement of Theorem \ref{fujita} involves only smooth varieties, in the proof we will need to work with a special class of singular varieties. Let us recall some definitions from the minimal model program. 

For us a pair $(X,\Delta)$ is always the data of a normal projective variety $X$ defined over $\cc$ and a $\rr$-Weil divisor $\Delta$ such that $K_X +\Delta$ is $\rr$-Cartier.  For any birational morphism $f: Y\rightarrow X$ from a normal variety $Y$, we can write 
\[
K_Y \sim_\rr f^*(K_X+\Delta) +\sum a(E,X,\Delta) E,
\]
where the sum runs over irreducible exceptional divisors of $f$ and the components of $f_* ^{-1}\Delta$. The number $a(E,X,\Delta)$ is called the discrepancy of $E$ with respect to the pair $(X,\Delta)$. We will allow some type of singularities of the pair.

\begin{definition}
Let $(X,\Delta)$ be a pair with coefficients of $\Delta$ in $[0,1]$. We say that $(X,\Delta)$ is $\epsilon$-lc if there exists a real number $\epsilon \geq 0$ such that for any log resolution of singularities $f: Y\rightarrow X$, the discrepancy $a(E,X,\Delta)\geq -1+\epsilon$ for any divisor $E$.
\end{definition}

We can easily recover the classical definitions of singularities in the minimal model program. For example, a $0$-lc pair is a lc pair and if $(X,\Delta)$ is $\epsilon$-lc for some $\epsilon >0$ then $(X,\Delta)$ is klt. 

We will use the minimal model program as established in \cite{BCHM}. More precisely, we will need the following result which is a combination of Theorem 1.2 and Corollary 1.3.3 in \cite{BCHM}. We refer to the original paper for the necessary definitions and the details of the proof.

\begin{theorem}\label{mmp}
Let $(X,\Delta)$ be a klt pair. Then
\begin{enumerate}
\item if $K_X+\Delta$ is pseudo-effective and $\Delta$ is big then $(X,\Delta)$ has a minimal model, or
\item if $K_X+\Delta$ is not pseudo-effective we can run a $K_X +\Delta$-MMP and end with a Mori fiber space.
\end{enumerate} 
\end{theorem}

It was already observed by Fujita, that Theorem \ref{fujita} is closely related to some boundedness statements. For example, he showed in \cite{Fujita} that the theorem is a consequence of the BAB conjecture. In \cite{DiC12}, we derived the ascending chain condition for $S_n$ using boundedness results on a suitable class of log Calabi-Yau pairs obtained in \cite{HMX}. In this paper, we will use the following theorem obtained recently by Birkar \cite{Bir16}.

\begin{theorem}\label{birkar}\cite[Theorem 1.4]{Bir16}
Fix a natural number $n$ and two positive real number $\epsilon$ and $\delta$. Consider pairs $(X,\Delta)$ such that
\begin{enumerate}
\item $(X,\Delta)$ is $\epsilon$-lc of dimension $n$, 
\item $\Delta$ is big with $K_{X}+\Delta \sim_\rr 0$, and 
\item the coefficients of $\Delta$ are more than or equal to $\delta$.
\end{enumerate} 
Then the set of such $X$ is a bounded family.
\end{theorem}

Recall that a set $\D$ of varieties is bounded is there exists a projective morphism $Z \rightarrow T$, where $T$ is of finite type, such that for any $X\in \D$ there exist a closed point in $t\in T$ and an isomorphism $Z_t \rightarrow X$.

\section{Fujita's spectrum conjecture}

In this section we prove Theorem \ref{fujita}. As we mentioned above we will need to work with a suitable set of singular pairs.  

\begin{definition}
Fix a positive integer $n$, a positive real number $\epsilon$ and a finite set $I\subset [0,1]$. We define $\D_{n}(\epsilon, I)$ to be the set of lc pairs $(X,\Delta)$ such that: 
\begin{enumerate}
\item $X$ is a normal projective variety of dimension $n$,
\item $\Delta$ is a big divisor with coefficients in $I$, and
\item $(X,t\Delta)$ is $\epsilon$-lc and $K_X +t \Delta$ is pseudo-effective for some $0\leq t \leq 1$. 
\end{enumerate}
\end{definition}

It was already observed in \cite{DiC12} that condition $(3)$ is necessary for our purposes. Moreover, Theorem \ref{dense} shows that we cannot avoid a condition involving the singularities of the pair. On the other hand, here we assume $\Delta$ to be a big divisor just to avoid a technical problem in the proof of Theorem \ref{structure}, see the remark after the proof, but it shouldn't be necessary.
  
We will study only pseudo-effective thresholds coming from pairs in $\D_{n}(\epsilon,I)$.

\begin{definition}
Fix a positive integer $n$, a positive real number $\epsilon$ and a finite set $I\subset [0,1]$. We define the sets 
\bdism
\T_{n}(\epsilon,I):=\left\{\tau(X,\Delta) \:|\: (X,\Delta )\in\D_{n}(\epsilon,I)\right\},
\edism
and 
\bdism
\T^{o}_{n}(\epsilon,I):=\left\{\tau(X,\Delta) \:|\: (X,\Delta) \in\D_{n}(\epsilon,I) \: \text{and} \: \rho(X)=1 \right\}.
\edism
\end{definition} 

A similar argument as in \cite{DiC12} provides a structure theorem for the set of pseudo-effective thresholds, see Theorem 1.6 there. On the other hand, the $\epsilon$-lc condition is quite tricky and we need to be more careful in the proof.

\begin{theorem}\label{structure}
Fix a positive integer $n$, a positive real number $\epsilon$ and $I\subset [0,1]$. Then
\[
\T_{n}(\epsilon,I)= \bigcup_{j=0}^{n} \T^{o}_{j}(\epsilon,I).
\] 
\end{theorem}

\begin{proof}
First, we will show that for any pair $(X,\Delta)\in\D_{n}(\epsilon,I)$ there exists another pair $(F,\Delta_F)\in\D_{j}(\epsilon,I)$ with $j\leq n$ and $\rho(F)=1$ such that $\tau=\tau(X,\Delta)=\tau(F,\Delta_F)$. By Corollary 1.37 in \cite{Kol13}, $(X,\Delta)$ admits a small $\qq$-factorial model in $\D_{n}(\epsilon,I)$ with same pseudo-effective threshold, since the map is small. In particular, we can assume that $X$ is a $\qq$-factorial variety. 

By definition $\D_{n}(\epsilon,I)$, we know that $(X,\tau \Delta)$ is $\epsilon$-lc, $K_X+\tau \Delta$ is pseudo-effective and $\Delta$ is big. In particular, by Theorem \ref{mmp}, we can run a $K_X+\tau \Delta$-MMP and end up with a minimal model. So we assume that $K_X +\tau \Delta$ is nef. 

For each $x<\tau$ we can run a $K_X +x \Delta$-MMP, again using Theorem \ref{mmp}, and obtain a birational map $f: X \dashrightarrow Y_x$, where $Y_x$ is Mori fiber space $Y_x \rightarrow Z_x$. Let $F_x$ be the generic fiber and $\Delta _{F_x}=( f_* \Delta )|_{F_x}$. By construction, we have that $x< \tau (F_x, \Delta_{F_x}) \leq \tau$. Each step of the minimal model program is $K_X +x \Delta$-negative and, if $x$ is sufficiently close to $\tau$, it is $K_X +\tau \Delta$-trivial. In other words, we are running a $K_X+x \Delta$-MMP with scaling of $\Delta$ since $K_X+\tau \Delta$ is nef. See \cite{Araujo} for a nice introduction on the MMP with scaling. The usual argument with the negativity lemma, as in \cite{KM}, implies that $a(E,Y_x,\tau f_*\Delta)\geq a(E,X, \tau \Delta)$, where $E$ is $f$-exceptional. In particular, we can assume that $(F, \tau \Delta_F)$ is $\epsilon$-lc as well. 

Let $\{x_i\}$ be an increasing sequence of rational numbers $x_i <\tau$ converging to $\tau$. For simplicity denote $F_i=F_{x_{i}}$ and similarly $\Delta_i=\Delta_{F_i}$. 
By the above argument, we can pass to a subsequence and assume that all the pairs $(F_i, \tau \Delta_i)$ are $\epsilon$-lc. Moreover, the coefficients of $x_i \Delta_{i}$ are in a DCC set since the coefficients of $\Delta_{i}$ are in a fixed finite set, recall that $\Delta$ is fixed. In particular, the sequence $\{\tau (F_i, \Delta_{F_i})\}$ in an increasing sequence converging to $\tau$ and by the ACC for pseudo-effective thresholds proved in \cite{DiC12}, we know that for $i$ large enough $\tau (F_i, \Delta_{F_i})=\tau(X,\Delta)$.

The other inclusion can be proved by taking $X\times E$, where $E$ is an elliptic curve. It is easy to show that $\tau(X,\Delta)=\tau(X\times E, \Delta\times E)$.
\end{proof}
 
\begin{remark}
The $\epsilon$-lc condition is preserved as long as the map $f$, constructed above, is $K_X+\tau \Delta$-trivial. In \cite{DiC12}, we avoided the problem using the ACC of log canonical thresholds. In this case the same proof will provide us that the fibers are lc, which is not enough for the purposes of this paper. Here, we added the condition that $\Delta$ is big, so that we can run a $K_X +\tau \Delta$-MMP and then deduce the necessary statement. Without that assumption, it is not clear that we can control uniformly the singularities of the fibers of the Mori fiber space. Note that a similar issue was pointed out also by Gongyo in \cite{Gongyo}.
\end{remark}

The above theorem implies that problems involving the pseudo-effective threshold of a pair can be reduced to varieties with Picard number $1$.

\begin{corollary}\label{zero}
Fix a positive integer $n$, two positive real numbers $\epsilon$ and $\eta$ and a finite set $I\subset [0,1]\cap \qq$. Then the set $\T^{o}_{n}(\epsilon,I) \cap [\eta, 1]$ is finite. 
\end{corollary}

\begin{proof}
Let $\tau=\tau(X,\Delta)\in \T^{o}_{n}(\epsilon,I) \cap [\eta, 1]$. As before, we can assume that $X$ is $\qq$-factorial. By assumption, $(X, \tau \Delta)$ is still $\epsilon$-lc, $K_X + \tau \Delta \equiv 0$ and the coefficients of $\tau \Delta$ are uniformly bounded from below, because $\tau\geq \eta>0$. By Theorem \ref{birkar}, the set of all such $X$'s is a bounded family. In particular, there is a natural number $i$, independent of $X$, such that $i K_{X}$ is Cartier. Moreover, we can choose a very ample divisor $A$ such that $-K_X \cdot A^{n-1} \leq M$ for some uniform $M$. In fact, we can choose $A$ to be a fixed multiple of $-iK_X$, again by boundedness. Since $K_{X}+\tau \Delta \equiv 0$, we have that $i (\tau \Delta) \cdot A^{n-1}$ is an integer uniformly bounded from above. Since, $\tau$ is bounded away from $0$, $\Delta \cdot A^{n-1}$ is uniformly bounded from above as well. Let $d$ be an integer such that $dx$ is an integer for any $x\in I$. In particular, $d\Delta$ is an integral divisor for any $\Delta$ appearing in $\T^{o}_{n}(\epsilon,I)$. We can choose $A^{n-1}$ to be represented by a cycle in the smooth locus of $X$ so that $d\Delta \cdot A^{n-1}$ is an integer. In particular, there exists a finite set $I'$, independent of $X$, such that $-K_X \cdot A ^{n-1} \in I'$ and $\Delta \cdot A^{n-1} \in I'$. This implies that 
\[
\tau=\frac{-K_X\cdot A^{n-1}}{\Delta \cdot A^{n-1}}
\]
must belong to a fixed finite set as well. 
\end{proof}

Combining Theorem \ref{structure} and the above result we obtain the following. 

\begin{corollary}\label{finite}
Fix a positive integer $n$, two positive real numbers $\epsilon$ and $\eta$ and a finite set $I\subset [0,1]\cap \qq$. Then 
$\T_{n}(\epsilon,I) \cap [\eta, 1]$ is a finite set.
\end{corollary}

We can now prove our main result. 

\begin{proof}[Proof of Theorem \ref{fujita}]
Suppose, by contradiction, that there exists a sequence of smooth polarized varieties $(X_i, H_i)$ such that the sequence $\{\tau(X_i,H_i)\}$ converges to a real number $\epsilon >0$. Since $H_i$ is ample, by Angehrn-Siu \cite{Ang} and the Cone Theorem, we know that $2(K_{X_i}+2n^{2} H_i)$ is base point free and ample. Let $H'_i \in |2(K_{X_i}+2n^{2} H_i)|$ be a general element such that $\left(X_i,\frac{1}{2}H'_i\right)$ is $1/2$-lc. Moreover, notice that 
\[
K_{X_i}+\frac{1}{2}H'_i \sim_\qq 2(K_{X_i}+n^{2}H_i)
\] 
is pseudo-effective, again by the Cone Theorem. Combining everything, we get that $\left(X_i,\frac{1}{2}H'_i\right)\in \D_n\left(\frac{1}{2},\{\frac{1}{2}\}\right)$. 
A simple computation shows that
\[
\tau\left(X_i,\frac{1}{2}H'_i \right)=2n^2 \frac{\tau(X_i, H_i)}{\tau(X_i, H_i)+1}.
\]

Thus, if the sequence $\{\tau(X_i, H_i)\}$ converges to a real number $\epsilon >0$, then the sequence $\{\tau\left(X_i, \frac{1}{2}H'_i\right)\}$ converges to a real number $\epsilon '>0$, which contradicts Corollary \ref{finite}.
\end{proof}

It would be quite useful to allow the divisor $H$ to be big and nef in Theorem \ref{fujita}. Unfortunately, the trick used in the proof above does not work in this more general setting unless we introduce some extra conditions. For many applications, for example see \cite{Yuri}, it is already interesting to work with Fano varieties, or more generally with varieties with $-K_X$ nef. We define $\mathcal{F}_n$ to be the set of pairs $(X,H)$ where
\begin{enumerate}
\item $X$ is a smooth projective variety of dimension $n$;
\item there exists an effective divisor $\Delta$ on $X$ such that $(X,\Delta)$ is klt and $-(K_X+\Delta)$ is nef, and
\item $H$ is a big and nef divisor on $X$.
\end{enumerate}

Note that we do not require any extra conditions on the coefficients of $\Delta$. Let 
\bdism
\mathcal{FT}_{n}:=\left\{\tau(X,H) \:|\: (X,H )\in\mathcal{F}_{n}\right\},
\edism
\begin{corollary}
Fix a positive integer $n$. Then $\mathcal{FT}_{n}\cap [\epsilon, \infty)$ is a finite set for any $\epsilon >0$.  
\end{corollary}

\begin{proof}
Since $H-(K_X+\Delta)$ is big and nef, by \cite{Kol2} there exists an integer $m$, depending only on $n$, such that $mH$ is base point free. Moreover, by \cite{Kol}, we can assume that $K_X+ 2mH$ is pseudo-effective. Let $H'$ be a general element of $|mH|$. Then the set of pairs $(X,\frac{1}{2}H')$ satifies Corollary \ref{finite}.  
\end{proof}

\section{Density of pseudo-effective thresholds}

In this section we will show that Theorem \ref{fujita} does not hold if we assume $H$ to be just a big divisor. The same argument will imply Theorem \ref{dense}. The construction we present here is based on Cutkosky's example of an effective divisor without Zariski decomposition. See \cite{Laz1} for more details. 

Let $C$ be an elliptic curve and $S:=C\times C$. It is well know that the cone of effective divisors of $S$ is a circular cone. Moreover, $\Eff(X)=\Nef(X)$. Let $C_{1}:=\left\{p\right\}\times C$, $C_{2}:=C\times \left\{p\right\}$ and let $\Gamma\subseteq S$ be the diagonal. Any $\rr$-divisor $D= x C_{1}+ y C_{2} + z \Gamma$ is nef if and only if 
\[
\begin{cases}
xy+xz+yz \geq 0, \\
x+y+z \geq 0.
\end{cases}
\]
Let $D$ and $H$ be two ample divisors. We define 
\bdism
\sigma(D,H):=\inf\left\{ t \in \rr_{\geq 0} \:|\: tD-H \:\text{is nef}\right\}.
\edism

The above equations make particularly easy the computation of $\sigma(D,H)$ once we know the coordinates of $D$ and $H$ in the cone generated by $C_{1}$, $C_{2}$ and $\Gamma$.

Cutkosky's idea is to construct threefolds with negative canonical divisor out of this situation. Let $X:=\pp(\oo_{S}\oplus\oo_{S}(-H))$. We can view $S\subseteq X$ as the zero section of the natural projection $\pi:X\rightarrow S$. By construction we have that $K_{X}=-2S + \pi^{*}(-H)$. Moreover, the cone $\Eff(X)$ is generated by $S$ and by $\pi^{*}\Eff(S)$.

We can now start the construction of the examples. The first thing to do is to find a sequence of divisors $(D_{k},H_{k})$ such that $\sigma(D_{k},H_{k})$ is an increasing sequence converging to a fixed number. Define 
\begin{align}\notag
D_{k}&:=k C_{1} +(k+1)C_{2}+(k+1)\Gamma, \\ \notag
H_{k}&:=k(C_{1}+C_{2}+\Gamma).
\end{align}

Note that these divisors are ample for any $k\geq 1$. It is not difficult to compute $\sigma(D_{k},H_{k})$.

\begin{lemma}
Let $D_{k}$ and $H_{k}$ be as above. Then
\bdism
\sigma(D_{k},H_{k})=\frac{3k}{3k+1}.
\edism
\end{lemma}

\begin{proof}
By definition of $\sigma(D_{k},H_{k})$, we need to find when the divisor $tD_{k}-H_{k}$ hits the boundary of the nef cone. Since 
\bdism
tD_{k}-H_{k}=k(t-1) C_{1}+(t(k+1)-k)C_{2}+(t(k+1)-k)\Gamma,
\edism 
the first equation of the nef cone implies that $tD_{k}-H_{k}$ is one of the root of the polynomial $(k(t-1)+t)(3k(t-1)+t)=0$. The smaller root $t_{0}=\frac{k}{k+1}$ does not give a nef divisor since the sum of the coordinates of $t_{0}D_{k}-H_{k}$ is negative. In particular, $\sigma(D_{k},H_{k})$ is given by the largest root of the polynomial, which in this case is $\frac{3k}{3k+1}$.
\end{proof}

We choose $D_{k}$ and $H_{k}$ as above so that $\sigma(D_{k},H_{k})$ has this simple form. Actually most of the choices of $D_{k}$ and $H_{k}$ would lead to an infinite sequence of pseudo-effective thresholds as long as $D_{k}\geq H_{k}$ and they ``converge" asymptotically. Moreover, a general choice will give us an irrational number $\sigma(D,H)$, since it is a root of a degree $2$ polynomial.

Let $X_{k}:=\pp(\oo_{S}\oplus\oo_{S}(-H_{k}))$ and let $M_{k}:=3S+\pi^{*}D_{k}$. Note that $M_{k}$ is big but not nef. By the fact that $\Eff(X)$ is generated by $S$ and by $\pi^{*}\Eff(S)$, we have that for $k\geq 3$ 
\[
\tau(X_{k},M_{k})=\sigma(D_{k},H_{k})=\frac{3k}{3k+1}.
\]

Thus, we can exhibit a sequence of pseudo-effective thresholds converging to $1$. Now Theorem \ref{dense} can be easily proved. 

\begin{proof}[Proof of Theorem \ref{dense}]
Let $a>0$ be any rational number. Let $X_k$ and $M_k$ as above. Choose $M'_k$ such that $M_k \sim a M'_k$. The same computation as above shows that, for $k$ large enough, the sequence $\tau(X_{k},M'_{k})$ converges to $a$.

Taking the product of $X_k$ and an elliptic curve, as we did in Theorem \ref{fujita}, we can show that the result holds in any dimension greater than $3$ as well. 
\end{proof}

It is of interest to compute $\vol(K_{X_{k}}+M_{k})$ as a function of $k$. Recall that the volume of a divisor $D$ is defined to be 
\[
\vol(D)= \lim _{m\to \infty} \frac{h^0(X,\oo_X (mD))}{m^n /n!}.
\]

\begin{lemma}
Let $X=\pp(\oo_{S}\oplus\oo_{S}(-H))$ and let $M=3S+\pi^{*}(D)$ a big divisor. Let $\sigma:=\sigma(D,H)$. Then 
\bdism
\vol(K_{X}+M)=3\int_{0}^{\frac{1-\sigma}{\sigma}}\left(D-(1+x)H\right)^{2} dx.
\edism 
\end{lemma}

\begin{proof}
By construction of $X$, we have that 
\begin{align} \notag
\pi_{*} \oo_{X} (m(K_{X}+M)) &=\pi_{*}\oo_{X}(mS+\pi^{*}(m(D-H)))  \\ \notag
&=\pi_{*}\oo_{X}(m)\otimes\oo_{S}(m(D-H)) \\ \notag
&=\bigoplus_{j=0}^{m}\oo_{S}(mD-(m+j)H).
\end{align}

Note that $h^{0}(S,\oo_{S}(mD-(m+j)H))=0$ if $j>\frac{1-\sigma}{\sigma}m$. Moreover, since $S$ is an abelian surface we have that
\bdism
h^{0}(X,\oo_{X} (m(K_{X}+M)))=\frac{1}{2}\sum_{j=0}^{\frac{1-\sigma}{\sigma}m}(mD-(m+j)H)^{2}.
\edism

Plugging in the above formula in the definition of the volume, we finally get
\begin{align} \notag
\vol(K_{X}+M)&=\lim_{m\to\infty}\frac{h^{0}(X,\oo_{X} (m(K_{X}+M)))}{m^{3}/3!} \\ \notag
&=\frac{3!}{2}\lim_{m\to\infty}\sum_{j=0}^{\frac{1-\sigma}{\sigma}m}\left(D-\left(1+\frac{j}{m}\right)H\right)^{2}\frac{1}{m} \\ \notag
&=3\int_{0}^{\frac{1-\sigma}{\sigma}}\left(D-(1+x)H\right)^{2} dx,
\end{align}

where the last equality follows by definition of Riemann sum.

\end{proof}

Applying the above lemma to our examples we get 
\bdism
\vol(K_{X_{k}}+M_{k})=3\int_{0}^{\frac{1}{k}}\left(D_{k}-(1+x)H_{k}\right)^{2} dx=\frac{2}{9k^3}-\frac{4}{3k^2}+\frac{6}{k}.
\edism
In particular, there is no effective birationality for those divisors since the volume is not uniformly bounded away from $0$. It would be interesting to show examples where the discreteness of the set of pseudo-effective thresholds fails but the coefficients of $M_k$ are in a fix bounded set. In dimension $2$ there are no such examples. In fact, as a consequence of the main result in \cite{Ale}, one can deduce the following.

\begin{corollary}\label{surf}
Let $I \in [0,1]$ be a DCC set. Then there is a positive real number $\delta_{I}$ such that if 
\begin{enumerate}
\item $X$ is a smooth surface,
\item the coefficients of $\Delta$ are in $I$, and 
\item $K_X +\Delta$ is big,
\end{enumerate}
then $\vol(K_{X}+\Delta)\geq \delta_{I}$.
\end{corollary}

It would be of interest to understand if $\tau(X,\Delta)$ is always a rational number where $(X,\Delta)$ satisfies the same assumptions of Corollary \ref{surf}.


\begin{thebibliography}{ELMNPM}

\bibitem[Ale94]{Ale} V. Alexeev, \textit{Boundedness and $K^2$ for log surfaces},
Internat. J. Math. 5 (1994), no. 6, 779-810. 

\bibitem[AS95]{Ang} U. Angehrn and Y.-T Siu, \textit{Effective freeness and point separation for adjoint bundles},
Invent. Math. \textbf{122} (1995), no. 2, 291-308.

\bibitem[Ara10]{Araujo} C. Araujo, \textit{The cone of pseudo-effective divisors on log varieties after Batyrev},
Math. Zeit. \textbf{264} (2010), no. 1, 179-193.

\bibitem[Bat98]{Batyrev} V. Batyrev and Y. Tschinkel, \textit{Tamagawa numbers of polarized algebraic varieties},
Ast\'erisque no. \textbf{251} (1998), 299-340. 

\bibitem[Bir16]{Bir16} C. Birkar, \textit{Anti-pluricanonical systems on Fano varieties}, arXiv:1603.05765 [math.AG], 2016.

\bibitem[BCHM10]{BCHM} C. Birkar, P. Cascini, C. Hacon and J. M$^{\text{c}}$Kernan,
\textit{Existence of minimal models for varieties of log general type},
J. Amer. Math. Soc. \textbf{23} (2010), 405-468.

\bibitem[DiC14]{Gabriele} G. Di Cerbo, \textit{Uniform bounds for the Iitaka fibration},
 Ann. Sc. Norm. Super. Pisa Cl. Sci. (5) \textbf{13} (2014), no. 4, 1133-1143. 

\bibitem[DiC12]{DiC12} G. Di Cerbo, \textit{On Fujita's log spectrum conjecture}, arXiv:1210.5324, (2012), to appear in Mathematische Annalen.

\bibitem[Fuj12]{Fujino} O. Fujino, \textit{Minimal model theory for log surfaces}, 
 Publ. Res. Inst. Math. Sci. \textbf{48} (2012), no. 2, 339-371. 

\bibitem[Fuj92]{Fujita1} T. Fujita, \textit{On Kodaira energy and adjoint reduction of polarized manifolds}, 
Manuscripta Math. \textbf{76} (1992), no. 1, 59-84.

\bibitem[Fuj95]{Fujita2} T. Fujita, \textit{On Kodaira energy and classification of polarized varieties}, 
Sugaku Expositions \textbf{8} (1995), no. 2, 183-196.

\bibitem[Fuj96]{Fujita} T. Fujita, \textit{On Kodaira energy of polarized log varieties}, 
J. Math. Soc. Japan \textbf{48}, (1996), no. 1, 1-12.

\bibitem[Gon15]{Gongyo} Y. Gongyo, \textit{Remarks on the non-vanishing conjecture}, 
Adv. Stud. Pure Math., \textbf{65}, Math. Soc. Japan, Tokyo, (2015), 107-116.

\bibitem[HMX14]{HMX} C. Hacon, J. M$^{\text{c}}$Kernan and C. Xu,
\textit{ACC for log canonical thresholds},  Ann. of Math. (2) \textbf{180} (2014), no. 2, 523-571.

\bibitem[Kol93]{Kol2} J. Koll\'ar, \textit{Effective base point freeness},
Math. Ann. \textbf{296} (1993), 595-605.

\bibitem[Kol97]{Kol} J. Koll\'ar, \textit{Singularities of pairs}, Algebraic Geometry, Santa Cruz 1995, Proc. Symp. Pure Math, vol. 62, Amer. Math. Soc., Providence, RI, 1997, 221-287.

\bibitem[Kol13]{Kol13} J. Koll\'ar, \textit{Singularities of the minimal model program},
with a collaboration of S\'andor Kov\'acs. Cambridge Tracts in Mathematics, 200. Cambridge University Press, Cambridge, 2013.

\bibitem[KM98]{KM} J. Koll\'ar and S. Mori, \textit{Birational geometry of algebraic varieties},
Cambridge Tracts in Mathematics, vol. 134, Cambridge University Press, 1998.

\bibitem[KLM13]{Kuronya} A. K\"uronya, V. Lozovanu and C. Maclean, \textit{Volume functions of linear series},
Math. Ann. 356 (2013), no. 2, 635-652. 

\bibitem[Laz04]{Laz1} R. Lazarsfeld, Positivity in Algebraic Geometry I,
Ergebnisse der Mathematik und ihrer Grenzgebiete. 3. Folge. A series
of Modern Survays in Mathematics \textbf{48},
\textit{Springer-Verlag, Berlin}, 2004.

\bibitem[Leh12]{Lehman} B. Lehmann, \textit{A cone theorem for nef curves}, 
J. Algebraic Geom. 21 (2012), no. 3, 473-493. 
 
\bibitem[Paf13]{Paf} A. Paffenholz, \textit{Finiteness of the polyhedral $\qq$-codegree spectrum}, 
Proc. Amer. Math. Soc. 143 (2015), no. 11, 4863-4873.
 
\bibitem[Tsc03]{Yuri} Y. Tschinkel, \textit{Fujita's program and rational points},
Higher dimensional varieties and rational points (Budapest, 2001), Bolyai Soc. Math. Stud., vol. 12,
Springer, Berlin, 2003, 283-310.

\end{thebibliography}
\end{document}